\documentclass[runningheads]{llncs}
\usepackage[T1]{fontenc}
\usepackage{graphicx}
\usepackage{lineno}
\usepackage{tcolorbox}
\usepackage{hyperref}
\usepackage{amsmath, amssymb}
\usepackage{cleveref}
\usepackage{booktabs}
\usepackage{color}

\usepackage{orcidlink}
\usepackage{pgfplots}

\usepackage{hyperref}
\hypersetup{
  colorlinks = true,
  citecolor = blue,
  linkcolor = blue,
  urlcolor = blue
}

\newcommand{\DS}{\mathsf{DS}}

\usepackage{cite}
\usepackage{tikz}
\usetikzlibrary{positioning, shapes.geometric, arrows.meta, calc}
\usetikzlibrary{backgrounds}
\usepackage{subcaption}

\usepackage{thm-restate}
\usepackage{comment}

\usepackage{upquote} 

\crefname{conjecture}{conjecture}{conjectures}
\Crefname{conjecture}{Conjecture}{Conjectures}
\crefname{question}{question}{questions}
\Crefname{question}{Question}{Questions}

\newcommand{\norm}[1]{\left\lVert #1 \right\rVert}
\newcommand{\alk}[1]{\textsf{AtLeast}_{#1}}

\begin{document}
\title{Doubly Saturated Ramsey Graphs: A Case Study in Computer-Assisted Mathematical Discovery}
\titlerunning{Doubly Saturated Ramsey Graphs}
%
\author{Benjamin Przybocki \orcidlink{0009-0007-5489-1733} \and
John Mackey \orcidlink{0000-0001-7319-4377} \and \\
Marijn J. H. Heule \orcidlink{0000-0002-5587-8801} \and Bernardo Subercaseaux \orcidlink{0000-0003-2295-1299}}
\authorrunning{Przybocki et al.}
%
\institute{Carnegie Mellon University, Pittsburgh, USA \\
\email{\{bprzyboc,jmackey,mheule,bsuberca\}@andrew.cmu.edu}}
\maketitle              
\begin{abstract}
Ramsey-good graphs are graphs that contain neither a clique of size $s$ nor an independent set of size $t$. We study doubly saturated Ramsey-good graphs, defined as Ramsey-good graphs in which the addition or removal of any edge necessarily creates an $s$-clique or a $t$-independent set. We present a method combining SAT solving with bespoke LLM-generated code to discover infinite families of such graphs, answering a question of Grinstead and Roberts from 1982. In addition, we use LLMs to generate and formalize correctness proofs in Lean. This case study highlights the potential of integrating automated reasoning, large language models, and formal verification to accelerate mathematical discovery. We argue that such tool-driven workflows will play an increasingly central role in experimental mathematics.

\end{abstract}
\section{Introduction}

In 2000, the mathematician Timothy Gowers presciently predicted that ``during the next century computers will become sufficiently good at proving theorems that the practice of pure mathematical research will be completely revolutionized''~\cite{gowers}. He looked forward to a ``golden age'' of mathematics, where computers can be helpful research assistants while still leaving room for humans to contribute deep insights. We believe that we are now entering this golden age, thanks to the combination of several technologies:
\begin{itemize}
    \item \textbf{Mature and efficient symbolic reasoning}: SAT solvers and computer algebra systems have evolved tremendously, and are currently able to reliably handle large computations.
    \item \textbf{LLMs capable of producing mathematical arguments}: Frontier models have demonstrated remarkable mathematical capabilities~\cite{yan2025surveymathematicalreasoningera}, ranging from competition problems~\cite{hubertOlympiadlevelFormalMathematical2026} to research-level questions and open problems~\cite{brenner2026solvingopenproblemtheoretical,gemini-math, feng2026semiautonomousmathematicsdiscoverygemini}.
    \item \textbf{Autoformalization}:
  Relying on proof assistants like Lean~\cite{lean4}, with large mathematical libraries~\cite{mathlib2020}, 
    LLMs can now formalize proofs for mathematical arguments of increasing complexity~\cite{urban2026130klinesformaltopology,weng2025autoformalizationeralargelanguage}.
\end{itemize}
In this paper, we give a glimpse of what this golden age looks like, and how these different technologies coalesce for accelerating mathematical research. Specifically, we show how our investigations into a natural problem in Ramsey theory were drastically assisted by the above technologies. Indeed, our theorems would have been very difficult to discover without computer assistance, and essentially all of the tedious work that inevitably arises in the course of research was fully automated.

\subsection{Problem and results}

A graph is \emph{$R(s,t)$-good} if it contains neither a clique of size $s$ nor an independent set of size $t$. Ramsey's celebrated theorem states that for each $s$ and $t$, there are only finitely many $R(s,t)$-good graphs~\cite{ramsey}. This theorem has been revisited and generalized many times and gave birth to a vast area of combinatorics called Ramsey theory~\cite{ramsey-theory-book}. In this paper, we revisit the topic once again. Let us say that a graph $G$ is \emph{doubly saturated $R(s,t)$-good} if it is $R(s,t)$-good, but adding or removing any edge from $G$ yields a graph that is not $R(s,t)$-good, and neither $G$ nor $\overline{G}$ is the complete graph.\footnote{The final condition ensures that it is possible to add and remove edges, which rules out, e.g., considering a graph on $1$ vertex as doubly saturated for all $s$ and $t$.} In other words, doubly saturated $R(s,t)$-good graphs are $R(s,t)$-good graphs that are edge-maximal and edge-minimal, excluding complete and empty graphs. What makes these graphs interesting is that they ``just barely'' avoid $s$-cliques and $t$-independent sets, so their construction involves a delicate balancing act. More broadly, our goal in understanding doubly saturated $R(s, t)$-good graphs is to understand better the space of $R(s, t)$-good graphs; as Angeltveit and McKay stated~\cite{Angeltveit2026}: ``The improvements in the upper bound for $R(5, 5)$ are intimately connected to our
understanding of [the set of $R(4, 5)$-good graphs on $n$ vertices] for $n$ large''. One way of understanding the space of $R(s, t)$-good graphs is to consider the partial ordering in which $G_1 \prec G_2$ if $G_2$ can be obtained from $G_1$ by adding edges. As illustrated in~\Cref{fig:hasse}, graphs can then be partitioned into weakly connected components of this ordering; doubly saturated graphs are precisely those that make up their entire connected component.

\begin{figure}
    \centering
    \input{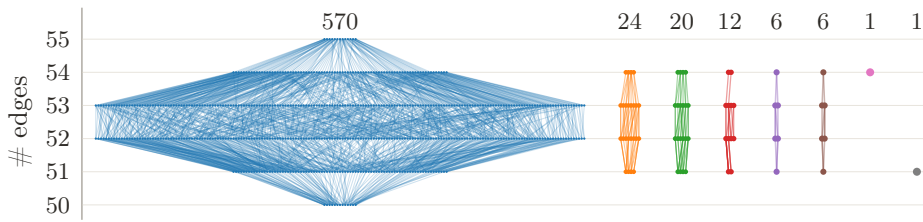}
    \caption{Hasse diagram for all 640 $R(4, 4)$-good graphs on $15$ vertices, including $2$ doubly saturated graphs. The size of each component is given above it.}
    \label{fig:hasse}
\end{figure}

According to Grinstead and Roberts~\cite{bicritical}, Albertson and Berman asked in 1980 whether there are any doubly saturated $R(s,t)$-good graphs other than $C_5$.\footnote{Grinstead and Roberts called these graphs \emph{bicritical}, but we avoid this terminology, since it is also used with a different meaning~\cite{bicritical-old}.} Grinstead and Roberts were able to find three others, all of which were circulant graphs, and they asked whether there were any others. Unfortunately, this natural problem has lain dormant since Grinstead and Roberts' paper.

With computer assistance, we prove that there are infinitely many circulant doubly saturated graphs:
\begin{restatable}{theorem}{thmrfour} \label{thm-r4t}
    For all $t \ge 4$, there is a doubly saturated $R(4,t)$-good graph on $6t-11$ vertices.
\end{restatable}

In fact, on the basis of experimental evidence, we conjecture that there are doubly saturated $R(s,t)$-good graphs for almost all choices of parameters:
\begin{conjecture} \label{conj-ds}
    Let $s,t \ge 3$ with $s \le t$ and $(s,t) \notin \{(3,4),(3,6)\}$. Then, there is a doubly saturated $R(s,t)$-good graph.
\end{conjecture}
In \Cref{sec-other-families}, we provide some evidence for this conjecture by verifying it for small choices of $s$ and $t$. Specifically, in \Cref{sec-r3t}, we verify the conjecture for $s = 3$ and $t \le 10$, and we propose a construction of an infinite family of graphs that we conjecture are doubly saturated $R(3,t)$-good. Then, in \Cref{sec-rss}, we use Paley graphs to show that there is a doubly saturated $R(s,s)$-good graph for all $s \in [3,20]$.

A simple computer search shows that there are no doubly saturated $R(3,4)$- or $R(3,6)$-good graphs.\footnote{This can easily be checked using Brendan McKay's enumeration of Ramsey-good graphs: \url{https://users.cecs.anu.edu.au/~bdm/data/ramsey.html}.} However, obtaining experimental data for larger choices of $s$ and $t$ requires efficient combinatorial optimization tooling. We use SAT solvers for this purpose, which allows us to prove for instance that the smallest doubly saturated $R(3, 7)$-good graph has $20$ vertices, and the smallest doubly saturated $R(4, 5)$-good graph has $19$ vertices. This relies on a compact SAT encoding with only $O(n^4)$ clauses for the double saturation constraint as well as symmetry-breaking constraints; we present this encoding in \Cref{sec-encoding}. In tandem with LLMs, we use these and similar computations to discover constructions of infinite families of doubly saturated graphs, including the one used to prove \Cref{thm-r4t}. This methodology is detailed in the next subsection.

\Cref{conj-ds} asks about the existence of doubly saturated $R(s,t)$-good graphs, but we can also ask more fine-grained questions about the set of doubly saturated $R(s,t)$-good graphs. One question that strikes us as particularly natural is how small a doubly saturated $R(s,t)$-good graph can be. In \Cref{sec-lb}, we present an LLM-discovered proof that every doubly saturated $R(s,t)$-good graph has at least $2s+2t-7$ vertices.

\subsection{Our computer-assisted methodology}\label{subsec:methodology}

\begin{figure}
    \centering
    \definecolor{ink}{RGB}{44, 52, 64}

\definecolor{stageone}{RGB}{49, 104, 142}
\definecolor{stageonefill}{RGB}{236, 243, 248}

\definecolor{stagetwo}{RGB}{44, 132, 125}
\definecolor{stagetwofill}{RGB}{233, 245, 242}

\definecolor{stagethree}{RGB}{184, 136, 43}
\definecolor{stagethreefill}{RGB}{249, 243, 228}

\definecolor{stagefour}{RGB}{186, 96, 65}
\definecolor{stagefourfill}{RGB}{250, 237, 232}

\definecolor{stagefive}{RGB}{124, 95, 156}
\definecolor{stagefivefill}{RGB}{242, 237, 247}

\begin{tikzpicture}[
    scale=0.68,
    transform shape,
    font=\large,
    basebox/.style={
        rectangle,
        rounded corners=2pt,
        draw=ink!75,
        line width=1.15pt,
        align=center,
        text=ink,
        text width=4.5cm,
        minimum height=2.0cm,
        inner sep=4pt
    },
    badge/.style={
        circle,
        minimum size=5.0mm,
        inner sep=0pt,
        anchor=north west,
        line width=0.9pt,
        text=white,
        font=\bfseries\small
    },
    arrow/.style={
        -{Stealth[length=3mm, width=2.6mm]},
        line width=1.35pt,
        draw=ink!68,
        line cap=round,
        line join=round,
        shorten <=1pt,
        shorten >=1pt
    }
]

\node[basebox, draw=stageone!85!ink, fill=stageonefill]
    (step1) at (0, 0) {
    \textbf{Formulate math problem}
};
\node[badge, fill=stageone!92!ink] at ([xshift=5pt, yshift=-5pt] step1.north west) {1};

\node[basebox, draw=stagetwo!85!ink, fill=stagetwofill]
    (step2) at (3.2, -3) {
    \textbf{Generate small-$n$ data} \\
    {\small (e.g., via SAT solvers)}
};
\node[badge, fill=stagetwo!92!ink] at ([xshift=5pt, yshift=-5pt] step2.north west) {2};

\node[basebox, draw=stagethree!88!ink, fill=stagethreefill]
    (step3) at (6.4, 0) {
    \vspace{0.1em}
    \textbf{Formulate conjecture} \\
    {\small (in tandem with LLMs)}
};
\node[badge, fill=stagethree!92!ink] at ([xshift=5pt, yshift=-5pt] step3.north west) {3};

\node[basebox, draw=stagefour!85!ink, fill=stagefourfill]
    (step4) at (9.6, -3) {
    \textbf{Prove conjecture} \\
    {\small (using LLMs)}
};
\node[badge, fill=stagefour!92!ink] at ([xshift=5pt, yshift=-5pt] step4.north west) {4};

\node[basebox, draw=stagefive!85!ink, fill=stagefivefill]
    (step5) at (12.8, 0) {
    \textbf{Formalize proof} \\
    {\small (via autoformalization system)}
};
\node[badge, fill=stagefive!92!ink] at ([xshift=5pt, yshift=-5pt] step5.north west) {5};

\draw[arrow] (step1) -- (step2);
\draw[arrow] (step2) -- (step3);
\draw[arrow] (step3) -- (step4);
\draw[arrow] (step4) -- (step5);

\end{tikzpicture}
    \caption{Computer-assisted mathematics research paradigm}
    \label{fig:paradigm}
\end{figure}
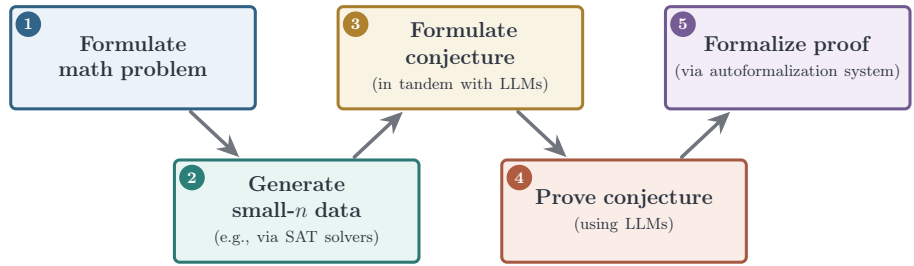

In this paper, we describe not only our mathematical results but also the methodology by which we discovered them, since we believe that workflows like ours will rapidly become the norm in experimental mathematics and will be a boon to the field. Our computer-assisted research paradigm is illustrated in \Cref{fig:paradigm}. Of course, these steps (except perhaps formalization) are part of the research process of almost every math problem, although it is only thanks to recent advances in AI and formal methods that multiple stages of this pipeline can be automated. Specifically, for step 2 of the process, we used SAT solvers to look for doubly saturated Ramsey-good graphs; in some cases, SAT solvers are not well suited to the problem, in which case we used bespoke LLM-generated code as a substitute. For about a decade now, it has been recognized that SAT solvers have tremendous potential for solving small cases of hard combinatorial problems~\cite{discrepancy,pythag,schur,keller,packing}. However, a common criticism of such proofs is that they do not scale to proving general results, partly because ``mere computations'' cannot provide the conceptual understanding that would be required for a general proof (see the discussion in \cite{varieties}). But this criticism is unwarranted in general: SAT solving (and computational experiments more generally) should be seen as an initial stage of the discovery process, not the end goal. This motivates step 3, where we studied the doubly saturated graphs found computationally, searched for patterns, and conjectured a general construction. We conjectured \Cref{thm-r4t} in the course of a dialogue with ChatGPT. This dialogue felt like a productive collaboration with a research partner, which was expedited by the fact that ChatGPT can autonomously run computational experiments to quickly assess hypotheses and analyze data. The role of the human in this dialogue was to guide ChatGPT toward promising strategies.

For step 4, we used LLMs to generate a proof of \Cref{thm-r4t} with essentially no human guidance. It is well known that LLMs have achieved remarkable performance on benchmarks and math competitions, but they are also starting to become useful for proving theorems arising in the course of actual research. Unfortunately, LLMs are notorious for producing proofs with fiendishly subtle errors. Thus, formalizing proofs in a proof assistant like Lean has become more important than ever. The primary barrier to the mainstream adoption of proof assistants has historically been their steep learning curve, together with the laborious nature of formalizing nontrivial mathematics. But here too, advances in AI and formal methods have changed the game. In our case, using an informal proof generated by Gemini, Harmonic's autoformalization system called Aristotle~\cite{achim2025aristotleimolevelautomatedtheorem} was able to autonomously write over 1000 lines of Lean code to formally verify \Cref{thm-r4t}.

\subsection{Related work}

\paragraph{Mathematical related work.} Saturation problems in graph theory have been studied since Erd\H{o}s, Hajnal, and Moon's classic paper~\cite{ehm}, which studied graphs for which the addition of any edge creates a $k$-clique. Specifically, they determined the minimum number of edges in an $n$-vertex graph avoiding $k$-cliques such that the addition of any edge creates a $k$-clique. These graphs were further studied by Hajnal~\cite{hajnal}, and we use one of Hajnal's theorems in \Cref{sec-lb}. Similar saturation problems have been studied with respect to subgraphs other than cliques~\cite{saturation-minimal,saturation-survey} and with respect to induced subgraphs~\cite{induced-saturation}. Damásdi et al.~\cite{saturation-ramsey} studied Ramsey-good graphs that are vertex-saturated, meaning that adding a vertex and connecting it to any subset of the existing vertices always yields a graph that is not Ramsey-good. The same paper also proved saturation results for other problems in Ramsey theory. More broadly, our work falls under the umbrella of studying objects that avoid some pattern but no longer do so after a small perturbation. The first author studied a problem in this vein in combinatorics on words, constructing words avoiding certain types of repetitions such that changing any letter creates such a repetition~\cite{przybocki}.

\paragraph{AI related work.} As mentioned above, SAT solvers have become an essential tool for computer-assisted mathematics, particularly for combinatorial problems~\cite{discrepancy,pythag,schur,keller,packing}. Most applications of SAT to mathematics focus on solving particular cases of open problems (step 2 from \Cref{fig:paradigm}), but some work has gone further and used data generated from SAT experiments to discover general theorems. For example, Subercaseaux et al.~\cite{pentagon-minimization} applied SAT to pentagon minimization in planar point sets and discovered a general construction, whose correctness they were able to prove by hand. Aside from SAT, using LLMs to assist with mathematical research has recently become a hot topic~\cite{gpt5-math,gemini-math,autonomous-math}. However, to the best of our knowledge, there are no case studies on the interplay between SAT and LLMs in spurring mathematical discovery, which is a significant gap given the recognized potential of neuro-symbolic AI~\cite{neurosymbolic}. The closest work to ours is by Xia et al.~\cite{xia2026agenticneurosymboliccollaborationmathematical}, which also uses a neuro-symbolic pipeline for a combinatorial design problem, integrating LLMs and Lean autoformalization. Their work does not include SAT solvers, instead using a custom Rust solver as well as computer algebra systems.

\subsection{Preliminaries}

All graphs in this paper are simple and finite. The \emph{complete graph} on $n$ vertices is denoted $K_n$, and the \emph{empty graph} on $n$ vertices is denoted $I_n$. Given a graph $G$, its \emph{clique number}, denoted $\omega(G)$, is the largest $n$ such that $G$ contains $K_n$ as a subgraph; its \emph{independence number}, denoted $\alpha(G)$, is the largest $n$ such that $G$ contains $I_n$ as an induced subgraph. The minimum (respectively, maximum) degree of vertices in $G$ is denoted $\delta(G)$ (respectively, $\Delta(G)$).

Given $x \in \mathbb{Z}/n\mathbb{Z}$, let
\(
    \norm{x} := \min(x \bmod n, n - (x \bmod n)).
\)
Then, a \emph{circulant graph} on $n$ vertices with distances $S \subseteq [\lfloor{n/2}\rfloor]$ is the graph with vertex set $\mathbb{Z}/n\mathbb{Z}$ such that $x$ and $y$ are adjacent if and only if $\norm{x-y} \in S$.

\section{A SAT encoding of double saturation} \label{sec-encoding}

In this section, we describe how to compactly encode the double saturation condition as a CNF formula that can be fed to a SAT solver. Given $s$, $t$, and $n$, we construct a CNF formula that is satisfiable if and only if there is a doubly saturated $R(s,t)$-good graph on $n$ vertices. Naturally, we have edge variables $e_{\{i,j\}}$ for each $\{i,j\} \in \binom{[n]}{2}$, which indicate whether the graph contains the edge $\{i,j\}$. To encode that the graph contains no $s$-clique, we use the following clauses:
\[
    \bigwedge_{S \in \binom{[n]}{s}} \left( \bigvee_{\{i,j\} \in \binom{S}{2}} \overline{e_{\{i,j\}}} \right).
\]
The absence of $t$-independent sets is encoded similarly. Note that this direct encoding of $R(s,t)$-goodness uses $\binom{n}{s} + \binom{n}{t}$ clauses.\footnote{An encoding with asymptotically fewer clauses using fast matrix multiplication is possible but unlikely to be practical~\cite{k-clique}.} To enforce that the graph is neither the complete graph nor the empty graph, we add the clauses $\bigvee_{i,j} e_{\{i,j\}}$ and $\bigvee_{i,j} \overline{e_{\{i,j\}}}$.

It remains to encode the double saturation condition, which amounts to enforcing:
\begin{itemize}
    \item[] (\textbf{Maximality})
    every non-edge $\{i,j\}$ is part of a $K_s \setminus \{\{i,j\}\}$
    \item[] (\textbf{Minimality}) and every edge $\{i,j\}$ is part of an $I_t \cup \{\{i,j\}\}$.
\end{itemize}
A naive encoding of these constraints would use $O(n^s + n^t)$ clauses for a fixed $s$ and $t$. Perhaps surprisingly, we show that the double saturation condition can be encoded using only $O(n^4)$ clauses. 
Let us describe the encoding for the maximality property, since minimality is analogous. For each $\{i,j\} \in \binom{[n]}{2}$ and each $k \in [n] \setminus \{i,j\}$, we introduce an auxiliary variable $p_{\{i,j\},k}$, which intuitively represents that vertex $k$ belongs to the $K_s \setminus \{\{i, j\}\}$ certifying maximality in case $\{i, j\}$ is a non-edge. Our encoding includes the following clauses:
\begin{align*}
    &\bigwedge_{\{i,j\} \in \binom{[n]}{2}} \bigwedge_{k \in [n] \setminus \{i,j\}} \left( e_{\{i,j\}} \lor \overline{p_{\{i,j\}, k}} \lor e_{\{i,k\}} \right) \land \left( e_{\{i,j\}} \lor \overline{p_{\{i,j\}, k}} \lor e_{\{j,k\}} \right) \\
   \land &\bigwedge_{\{i,j\} \in \binom{[n]}{2}} \bigwedge_{\{k,k'\} \in \binom{[n] \setminus \{i,j\}}{2}} \left( e_{\{i,j\}} \lor \overline{p_{\{i,j\},k}} \lor \overline{p_{\{i,j\},k'}} \lor e_{\{k,k'\}} \right) \\
   \land &\bigwedge_{\{i,j\} \in \binom{[n]}{2}} \alk{s-2}(\{p_{\{i,j\},k} \mid k \in [n] \setminus \{i,j\}\}).
\end{align*}
These constraints enforce that for every non-edge $\{i,j\}$, there are at least $s-2$ vertices that are all pairwise adjacent and adjacent to both $i$ and $j$, which is to say that $\{i,j\}$ is part of a $K_s \setminus \{\{i,j\}\}$. An illustration is presented in~\Cref{fig:encoding}. We encode $\alk{s-2}(\{p_{\{i,j\},k} \mid k \in [n] \setminus \{i,j\}\})$ with $O(ns)$ clauses using Sinz's sequential counter encoding~\cite{sinz}; we use the implementation of cardinality constraints provided by PySAT \cite{pysat}. Therefore, the number of clauses is $O(n^4)$.

\begin{figure}
    \begin{subfigure}{0.49\linewidth}
        \centering
        \begin{tikzpicture}
    \node[circle, fill=blue, text=white, minimum width=17pt, inner sep=1pt] (i) at (135:1.5) {$i$};
    \node[circle, fill=blue, text=white, minimum width=17pt, inner sep=1pt] (j) at (225:1.5) {$j$};
    \node[circle, fill=orange!50,  minimum width=17pt, inner sep=1pt] (k) at (315:1.5) {\scriptsize $k_2$};
    \node[circle, fill=orange!50,  minimum width=17pt, inner sep=1pt] (kp) at (45:1.5) {\scriptsize $k_1$};

    \draw[thick] (i) -- (k);
    \draw[thick] (i) -- (kp);

    \draw[thick] (j) -- (k);
    \draw[thick] (j) -- (kp);

    \draw[thick] (k) -- (kp);

    \draw[thick, dashed, red] (i) -- (j);
\end{tikzpicture}
        \caption{$s=4$}
    \end{subfigure}
    \begin{subfigure}{0.49\linewidth}
        \centering
        \begin{tikzpicture}[scale=0.67]
    \node[circle, fill=blue, text=white, minimum width=17pt, inner sep=0] (i) at (154.29:2.5) {$i$};
    \node[circle, fill=blue, text=white, minimum width=17pt,  inner sep=0] (j) at (205.71:2.5) {$j$};
    
    \node[circle, fill=orange!50, minimum width=17pt, inner sep=1pt] (k1) at (257.14:2.5) {\scriptsize $k_5$};
    \node[circle, fill=orange!50, minimum width=17pt, inner sep=1pt] (k2) at (308.57:2.5) {\scriptsize $k_4$};
    \node[circle, fill=orange!50, minimum width=17pt, inner sep=1pt] (k3) at (0:2.5) {\scriptsize $k_3$};
    \node[circle, fill=orange!50, minimum width=17pt, inner sep=1pt] (k4) at (51.43:2.5) {\scriptsize $k_2$};
    \node[circle, fill=orange!50, minimum width=17pt, inner sep=1pt] (k5) at (102.86:2.5) {\scriptsize $k_1$};

    \begin{scope}[on background layer]
        \foreach \x in {1,...,5} {
            \foreach \y in {1,...,5} {
                \draw[thick] (k\x) -- (k\y);
            }
            \draw[thick] (k\x) -- (i);
            \draw[thick] (k\x) -- (j);
        }
    \end{scope}
    
    \draw[thick, dashed, red] (i) -- (j);
    
\end{tikzpicture}
          \caption{$s=7$}
    \end{subfigure}
    \caption{Illustration of the encoding for maximality. If edge $\{i, j\}$ is not present, there must be vertices $k_1, \dots, k_{s-2}$ such that adding edge $\{i, j\}$ would introduce a $K_{s}$.}
    \label{fig:encoding}
\end{figure}
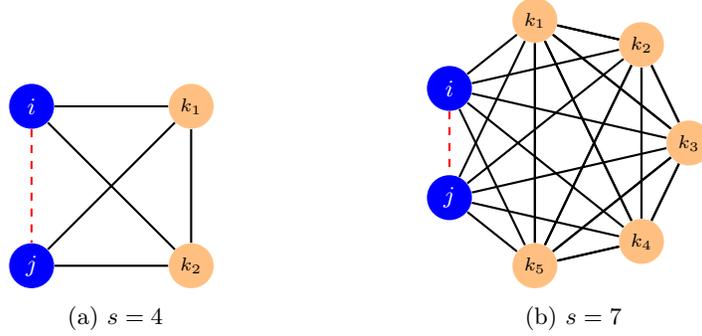

To make the SAT encoding performant, we also introduce lexicographic symmetry-breaking constraints~\cite[Definition~8]{codish}. Namely, we enforce for each $i < j$ that
\begin{align*}
    &(e_{\{i,1\}}, \dots, e_{\{i,i-1\}}, e_{\{i,i+1\}}, \dots, e_{\{i,j-1\}}, e_{\{i,j+1\}}, \dots, e_{\{i,n\}}) \preceq_{\text{lex}} \\
    &(e_{\{j,1\}}, \dots, e_{\{j,i-1\}}, e_{\{j,i+1\}}, \dots, e_{\{j,j-1\}}, e_{\{j,j+1\}}, \dots, e_{\{j,n\}}).
\end{align*}
We encode the $\preceq_{\text{lex}}$ constraint between two sequences $(a_1, \dots, a_n)$, $(b_1, \dots, b_n)$ of propositional variables via the following recursion. First, we add a clause $(\overline{a_1} \lor b_1)$, then introduce an auxiliary variable $y$, and add 4 clauses to enforce that $y \leftrightarrow (a_1 \leftrightarrow b_1)$. Finally, we recursively encode $(a_2, \dots, a_n) \preceq_{\text{lex}} (b_2, \dots, b_n)$ into a formula $\lambda_{n-1}$, and add clauses $(\overline{y} \lor C)$ for every $C \in \lambda_{n-1}$.

\section{Doubly saturated $R(4,t)$-good graphs} \label{sec-r4t}

Since there is a complete enumeration of $R(4,4)$-good graphs, it is a simple computation to check that there are six doubly saturated $R(4,4)$-good graphs: one on 13 vertices, two on 15 vertices, two on 16 vertices, and one on 17 vertices. The doubly saturated $R(4,4)$-good graph on 13 vertices is the Paley graph of order 13, which is the circulant graph with distances $\{1,3,4\}$.

Following these observations, we turned to looking for doubly saturated $R(4,5)$-good graphs. There are far too many $R(4,5)$-good graphs for a complete enumeration (the estimate of Angeltveit and McKay is about $2.93 \cdot 10^{19}$~\cite[Appendix A]{Angeltveit2026}), so we applied SAT to the problem using the encoding from \Cref{sec-encoding}. We found that there are no doubly saturated $R(4,5)$-good graphs on 18 or fewer vertices. On the other hand, for 19 vertices, a SAT solver (CaDiCaL~\cite{cadical}) does not terminate after a day on a single core. Thus, we made an ansatz that a solution would be circulant, as suggested by some of the $R(4,4)$ results. We found that there is a circulant doubly saturated $R(4,5)$-good graph on 19 vertices with distances $\{4,5,6,8\}$. Later, we applied Mallob~\cite{mallob} to our SAT encoding, and it found a solution isomorphic to this circulant graph, which suggests that it may be the unique doubly saturated $R(4,5)$-good graph on 19 vertices.

At this point, we asked ChatGPT-5.1 Pro to look for circulant doubly saturated $R(4,t)$-good graphs for $t \in [6,12]$. For large values of $t$, the SAT encoding becomes unwieldy because there are too many clauses required to enforce the absence of $t$-independent sets, but bespoke code generated by the LLM succeeded in finding circulant constructions for each $t \in [6,12]$.

Having found nine constructions, we asked ChatGPT to extrapolate from these to find a general construction that works for all $t \ge 4$. ChatGPT quickly recognized that every construction so far had $n := 6t-11$ vertices, so it remained to find the patterns in the distance sets. ChatGPT tried several approaches, which varied in how promising they seemed to us. For example, analyzing the spectral properties of the graphs did not seem to us to be as fruitful as studying the additive-combinatorial properties of the distance sets. Accordingly, we encouraged ChatGPT to look deeper into the more promising strategies. Besides this minimal shepherding, the discovery process was almost entirely LLM-driven. ChatGPT realized that when specifying a circulant graph, there are many distance sets that result in isomorphic graphs. Motivated by this observation, ChatGPT tried rewriting each of the nine distance sets into a canonical form from which patterns would become clearer. After a few iterations of this, it discovered the following construction for the distance sets, which agreed with all nine of our constructions up to isomorphism: $\{1\} \cup \{m \in [\lfloor n/2 \rfloor] \mid m \equiv 3,4 \pmod{6}\}$. At the time, ChatGPT was unable to prove that the construction works.

With the release of ChatGPT-5.4 Pro, we revisited the problem. This time, ChatGPT not only managed to produce a proof but also showed that the conjectured construction has a simpler isomorphic description. Specifically, letting $m = t-2$ and $n = 6t-11$, the above construction is isomorphic to the circulant graph on $n$ vertices with distances $\{m\} \cup [2m+1, 3m]$. See \Cref{fig:construction} for a depiction of the construction for $t=4$ and $t=5$.

We next set our sights on obtaining a formal proof that the construction works. We believe that this is important when using LLMs to prove theorems, because LLM-generated proofs are notorious for containing subtle errors, and formalization eliminates any worries about correctness. Moreover, autoformalization avoids wasting mathematicians' time on checking flawed LLM-generated proofs, and can be further used to refine proof attempts on a loop based on errors that the formalization process reveals.
To autoformalize a proof, it is important to start with an informal proof that is sufficiently detailed. We prompted Gemini 3 Deep Think\footnote{Switching to Gemini from ChatGPT was simply because each of us prefers a different LLM, not because we had any reason to expect Gemini to be better at this task.} to ``write a proof that is very detailed and formal so that it can readily be formalized in Lean.'' We fed this proof into Aristotle, which generated a complete formal proof in Lean after two attempts.\footnote{The formalization is available at \url{https://github.com/bprzybocki/doubly-saturated}.}

Below, we present an informal proof that the construction works. The proof has been human-edited for clarity and concision, particularly because the LLM-generated proof had some needless case distinctions. Since the formal proof contains all of the details, we write the informal proof somewhat tersely.

\thmrfour*
\begin{proof}
Let $m = t - 2$, meaning the number of vertices is $n = 6m + 1 = 6t - 11$. Let $G$ be the circulant graph on $n$ vertices with distances $\{m\} \cup [2m+1, 3m]$. Thus, $u$ and $v$ are adjacent if and only if $u-v \in \{m, 5m+1\} \cup [2m+1, 4m]$.

To show $G$ is doubly saturated $R(4, t)$-good, we must establish four properties: $\omega(G) < 4$, $\alpha(G) < t$, adding any edge creates a $K_4$, and removing any edge creates an $I_t$.

\textbf{Part 1: $\omega(G) < 4$.} Assume for contradiction that a clique of $4$ vertices exists. The forward cyclic differences (gaps) $g_1, g_2, g_3, g_4$ between adjacent vertices must satisfy $g_1 + g_2 + g_3 + g_4 = n = 6m+1$. At least three gaps must equal $m$, otherwise $g_1 + g_2 + g_3 + g_4 \ge 2m + 2 \cdot (2m+1) > 6m+1 = n$. Thus, there is some $i \in [4]$ for which we have $g_i = g_{i+1} = m$ (where $g_5 = g_1$). But then, $g_i + g_{i+1} = 2m \notin \{m, 5m+1\} \cup [2m+1, 4m]$, a contradiction.

\textbf{Part 2: $\alpha(G) < t$.} Assume for contradiction that an independent set of $t=m+2$ vertices exists. Let the vertices be $v_1, \dots, v_t$. Then, $\norm{v_i - v_j} \in [0,m-1] \cup [m+1,2m]$ for all $i,j \in [t]$. Since $\norm{v_i - v_j} \le 2m < n/3$, we can cyclically shift the vertices so that $v_1 = 0$ and $v_i \in [1,m-1] \cup [m+1,2m]$ for all $i \in [2,t]$. But since $\{v_2,\dots,v_t\}$ is a subset of $[1,m-1] \cup [m+1,2m]$ of size $m+1$, there must be some pair $v_i$ and $v_j$ such that $\norm{v_i - v_j} = m$, a contradiction.

\textbf{Part 3: Adding any edge creates a $K_4$.} Suppose we add an edge $(0, d)$ where $d \in [1, m-1] \cup [m+1, 2m]$. First, if $d \in [1,m-1]$, then $\{0,d,2m+d+1,3m+d+1\}$ is a $K_4$. Second, if $d \in [m+1,2m]$, then $\{0,d,m+d,5m+1\}$ is a $K_4$.

\textbf{Part 4: Removing any edge creates an $I_t$.} Suppose we remove an edge $(0, d)$ where $d \in S$. First, if $d = m$, then $[-1,m-2] \cup \{m,2m-1\}$ is an $I_t$. Second, if $d \in [2m+1,3m-1]$, then $\{0,2m,d\} \cup ([d-2m, d-m-1] \setminus \{m\})$ is an $I_t$. Third, if $d = 3m$, then $[-2m,-m-1] \cup \{0,3m\}$ is an $I_t$.
\end{proof}

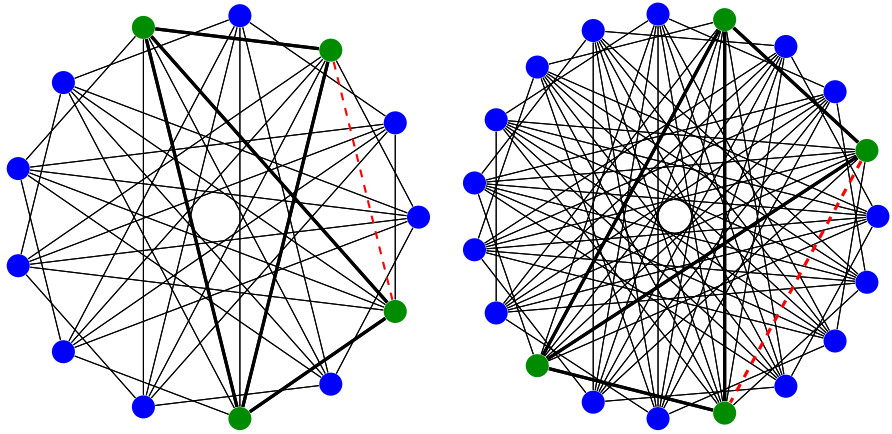
\begin{figure}
    \begin{subfigure}{0.49\linewidth}
        \centering
    \begin{tikzpicture}[scale=0.9]
\node[circle, fill=blue, text=white] (v0) at (3.0, 0.0) {};
\node[circle, fill=blue, text=white] (v1) at (2.65636807695963, 1.3941695161313055) {};
\node[circle, fill=green!55!black, text=white] (v2) at (1.7041942401934675, 2.4689515976809693) {};
\node[circle, fill=blue, text=white] (v3) at (0.36161004076596903, 2.978126622294162) {};
\node[circle, fill=green!55!black, text=white] (v4) at (-1.0638146611276063, 2.8050487280562444) {};
\node[circle, fill=blue, text=white] (v5) at (-2.2455322445133037, 1.9893679747223856) {};
\node[circle, fill=blue, text=white] (v6) at (-2.912825452278156, 0.717946992862673) {};
\node[circle, fill=blue, text=white] (v7) at (-2.9128254522781565, -0.7179469928626723) {};
\node[circle, fill=blue, text=white] (v8) at (-2.2455322445133037, -1.989367974722385) {};
\node[circle, fill=blue, text=white] (v9) at (-1.0638146611276076, -2.805048728056244) {};
\node[circle, fill=green!55!black, text=white] (v10) at (0.3616100407659696, -2.978126622294162) {};
\node[circle, fill=blue, text=white] (v11) at (1.7041942401934644, -2.468951597680971) {};
\node[circle, fill=green!55!black, text=white] (v12) at (2.6563680769596303, -1.3941695161313052) {};
\draw[black] (v0) -- (v2);
\draw[black] (v0) -- (v5);
\draw[black] (v0) -- (v6);
\draw[black] (v0) -- (v7);
\draw[black] (v0) -- (v8);
\draw[black] (v0) -- (v11);
\draw[black] (v1) -- (v3);
\draw[black] (v1) -- (v6);
\draw[black] (v1) -- (v7);
\draw[black] (v1) -- (v8);
\draw[black] (v1) -- (v9);
\draw[black] (v1) -- (v12);
\draw[black] (v2) -- (v0);
\draw[very thick] (v2) -- (v4);
\draw[black] (v2) -- (v7);
\draw[black] (v2) -- (v8);
\draw[black] (v2) -- (v9);
\draw[very thick] (v2) -- (v10);
\draw[black] (v3) -- (v1);
\draw[black] (v3) -- (v5);
\draw[black] (v3) -- (v8);
\draw[black] (v3) -- (v9);
\draw[black] (v3) -- (v10);
\draw[black] (v3) -- (v11);
\draw[very thick] (v4) -- (v2);
\draw[black] (v4) -- (v6);
\draw[black] (v4) -- (v9);
\draw[very thick] (v4) -- (v10);
\draw[black] (v4) -- (v11);
\draw[very thick] (v4) -- (v12);
\draw[black] (v5) -- (v0);
\draw[black] (v5) -- (v3);
\draw[black] (v5) -- (v7);
\draw[black] (v5) -- (v10);
\draw[black] (v5) -- (v11);
\draw[black] (v5) -- (v12);
\draw[black] (v6) -- (v0);
\draw[black] (v6) -- (v1);
\draw[black] (v6) -- (v4);
\draw[black] (v6) -- (v8);
\draw[black] (v6) -- (v11);
\draw[black] (v6) -- (v12);
\draw[black] (v7) -- (v0);
\draw[black] (v7) -- (v1);
\draw[black] (v7) -- (v2);
\draw[black] (v7) -- (v5);
\draw[black] (v7) -- (v9);
\draw[black] (v7) -- (v12);
\draw[black] (v8) -- (v0);
\draw[black] (v8) -- (v1);
\draw[black] (v8) -- (v2);
\draw[black] (v8) -- (v3);
\draw[black] (v8) -- (v6);
\draw[black] (v8) -- (v10);
\draw[black] (v9) -- (v1);
\draw[black] (v9) -- (v2);
\draw[black] (v9) -- (v3);
\draw[black] (v9) -- (v4);
\draw[black] (v9) -- (v7);
\draw[black] (v9) -- (v11);
\draw[very thick] (v10) -- (v2);
\draw[black] (v10) -- (v3);
\draw[very thick] (v10) -- (v4);
\draw[black] (v10) -- (v5);
\draw[black] (v10) -- (v8);
\draw[very thick] (v10) -- (v12);
\draw[black] (v11) -- (v0);
\draw[black] (v11) -- (v3);
\draw[black] (v11) -- (v4);
\draw[black] (v11) -- (v5);
\draw[black] (v11) -- (v6);
\draw[black] (v11) -- (v9);
\draw[black] (v12) -- (v1);
\draw[very thick] (v12) -- (v4);
\draw[black] (v12) -- (v5);
\draw[black] (v12) -- (v6);
\draw[black] (v12) -- (v7);
\draw[very thick] (v12) -- (v10);
\draw[red, dashed, thick] (v2) -- (v12);
\end{tikzpicture}
    \end{subfigure}
    \begin{subfigure}{0.49\linewidth}
        \centering
\begin{tikzpicture}[scale=0.9]
\node[circle, fill=blue, text=white] (v0) at (3.0, 0.0) {};
\node[circle, fill=green!55!black, text=white] (v1) at (2.837451725101904, 0.9740984076140504) {};
\node[circle, fill=blue, text=white] (v2) at (2.3674215281891806, 1.8426381380690033) {};
\node[circle, fill=blue, text=white] (v3) at (1.6408444743672808, 2.5114994347875856) {};
\node[circle, fill=green!55!black, text=white] (v4) at (0.7364564614223977, 2.908200797817991) {};
\node[circle, fill=blue, text=white] (v5) at (-0.24773803641699682, 2.9897534790200098) {};
\node[circle, fill=blue, text=white] (v6) at (-1.2050862739589083, 2.747319979965172) {};
\node[circle, fill=blue, text=white] (v7) at (-2.0318447148772227, 2.2071717320193955) {};
\node[circle, fill=blue, text=white] (v8) at (-2.638421253619467, 1.427842179111221) {};
\node[circle, fill=blue, text=white] (v9) at (-2.959083910208167, 0.4937837708422021) {};
\node[circle, fill=blue, text=white] (v10) at (-2.959083910208167, -0.4937837708422014) {};
\node[circle, fill=blue, text=white] (v11) at (-2.6384212536194678, -1.4278421791112192) {};
\node[circle, fill=green!55!black, text=white] (v12) at (-2.031844714877223, -2.2071717320193946) {};
\node[circle, fill=blue, text=white] (v13) at (-1.2050862739589072, -2.747319979965173) {};
\node[circle, fill=blue, text=white] (v14) at (-0.2477380364169982, -2.9897534790200098) {};
\node[circle, fill=green!55!black, text=white] (v15) at (0.7364564614223964, -2.9082007978179916) {};
\node[circle, fill=blue, text=white] (v16) at (1.6408444743672796, -2.511499434787586) {};
\node[circle, fill=blue, text=white] (v17) at (2.3674215281891815, -1.842638138069002) {};
\node[circle, fill=blue, text=white] (v18) at (2.837451725101904, -0.9740984076140512) {};
\draw[black] (v0) -- (v3);
\draw[black] (v0) -- (v7);
\draw[black] (v0) -- (v8);
\draw[black] (v0) -- (v9);
\draw[black] (v0) -- (v10);
\draw[black] (v0) -- (v11);
\draw[black] (v0) -- (v12);
\draw[black] (v0) -- (v16);
\draw[very thick] (v1) -- (v4);
\draw[black] (v1) -- (v8);
\draw[black] (v1) -- (v9);
\draw[black] (v1) -- (v10);
\draw[black] (v1) -- (v11);
\draw[very thick] (v1) -- (v12);
\draw[black] (v1) -- (v13);
\draw[black] (v1) -- (v17);
\draw[black] (v2) -- (v5);
\draw[black] (v2) -- (v9);
\draw[black] (v2) -- (v10);
\draw[black] (v2) -- (v11);
\draw[black] (v2) -- (v12);
\draw[black] (v2) -- (v13);
\draw[black] (v2) -- (v14);
\draw[black] (v2) -- (v18);
\draw[black] (v3) -- (v0);
\draw[black] (v3) -- (v6);
\draw[black] (v3) -- (v10);
\draw[black] (v3) -- (v11);
\draw[black] (v3) -- (v12);
\draw[black] (v3) -- (v13);
\draw[black] (v3) -- (v14);
\draw[black] (v3) -- (v15);
\draw[very thick] (v4) -- (v1);
\draw[black] (v4) -- (v7);
\draw[black] (v4) -- (v11);
\draw[very thick] (v4) -- (v12);
\draw[black] (v4) -- (v13);
\draw[black] (v4) -- (v14);
\draw[very thick] (v4) -- (v15);
\draw[black] (v4) -- (v16);
\draw[black] (v5) -- (v2);
\draw[black] (v5) -- (v8);
\draw[black] (v5) -- (v12);
\draw[black] (v5) -- (v13);
\draw[black] (v5) -- (v14);
\draw[black] (v5) -- (v15);
\draw[black] (v5) -- (v16);
\draw[black] (v5) -- (v17);
\draw[black] (v6) -- (v3);
\draw[black] (v6) -- (v9);
\draw[black] (v6) -- (v13);
\draw[black] (v6) -- (v14);
\draw[black] (v6) -- (v15);
\draw[black] (v6) -- (v16);
\draw[black] (v6) -- (v17);
\draw[black] (v6) -- (v18);
\draw[black] (v7) -- (v0);
\draw[black] (v7) -- (v4);
\draw[black] (v7) -- (v10);
\draw[black] (v7) -- (v14);
\draw[black] (v7) -- (v15);
\draw[black] (v7) -- (v16);
\draw[black] (v7) -- (v17);
\draw[black] (v7) -- (v18);
\draw[black] (v8) -- (v0);
\draw[black] (v8) -- (v1);
\draw[black] (v8) -- (v5);
\draw[black] (v8) -- (v11);
\draw[black] (v8) -- (v15);
\draw[black] (v8) -- (v16);
\draw[black] (v8) -- (v17);
\draw[black] (v8) -- (v18);
\draw[black] (v9) -- (v0);
\draw[black] (v9) -- (v1);
\draw[black] (v9) -- (v2);
\draw[black] (v9) -- (v6);
\draw[black] (v9) -- (v12);
\draw[black] (v9) -- (v16);
\draw[black] (v9) -- (v17);
\draw[black] (v9) -- (v18);
\draw[black] (v10) -- (v0);
\draw[black] (v10) -- (v1);
\draw[black] (v10) -- (v2);
\draw[black] (v10) -- (v3);
\draw[black] (v10) -- (v7);
\draw[black] (v10) -- (v13);
\draw[black] (v10) -- (v17);
\draw[black] (v10) -- (v18);
\draw[black] (v11) -- (v0);
\draw[black] (v11) -- (v1);
\draw[black] (v11) -- (v2);
\draw[black] (v11) -- (v3);
\draw[black] (v11) -- (v4);
\draw[black] (v11) -- (v8);
\draw[black] (v11) -- (v14);
\draw[black] (v11) -- (v18);
\draw[black] (v12) -- (v0);
\draw[very thick] (v12) -- (v1);
\draw[black] (v12) -- (v2);
\draw[black] (v12) -- (v3);
\draw[very thick] (v12) -- (v4);
\draw[black] (v12) -- (v5);
\draw[black] (v12) -- (v9);
\draw[very thick] (v12) -- (v15);
\draw[black] (v13) -- (v1);
\draw[black] (v13) -- (v2);
\draw[black] (v13) -- (v3);
\draw[black] (v13) -- (v4);
\draw[black] (v13) -- (v5);
\draw[black] (v13) -- (v6);
\draw[black] (v13) -- (v10);
\draw[black] (v13) -- (v16);
\draw[black] (v14) -- (v2);
\draw[black] (v14) -- (v3);
\draw[black] (v14) -- (v4);
\draw[black] (v14) -- (v5);
\draw[black] (v14) -- (v6);
\draw[black] (v14) -- (v7);
\draw[black] (v14) -- (v11);
\draw[black] (v14) -- (v17);
\draw[black] (v15) -- (v3);
\draw[very thick] (v15) -- (v4);
\draw[black] (v15) -- (v5);
\draw[black] (v15) -- (v6);
\draw[black] (v15) -- (v7);
\draw[black] (v15) -- (v8);
\draw[very thick] (v15) -- (v12);
\draw[black] (v15) -- (v18);
\draw[black] (v16) -- (v0);
\draw[black] (v16) -- (v4);
\draw[black] (v16) -- (v5);
\draw[black] (v16) -- (v6);
\draw[black] (v16) -- (v7);
\draw[black] (v16) -- (v8);
\draw[black] (v16) -- (v9);
\draw[black] (v16) -- (v13);
\draw[black] (v17) -- (v1);
\draw[black] (v17) -- (v5);
\draw[black] (v17) -- (v6);
\draw[black] (v17) -- (v7);
\draw[black] (v17) -- (v8);
\draw[black] (v17) -- (v9);
\draw[black] (v17) -- (v10);
\draw[black] (v17) -- (v14);
\draw[black] (v18) -- (v2);
\draw[black] (v18) -- (v6);
\draw[black] (v18) -- (v7);
\draw[black] (v18) -- (v8);
\draw[black] (v18) -- (v9);
\draw[black] (v18) -- (v10);
\draw[black] (v18) -- (v11);
\draw[black] (v18) -- (v15);
\draw[red, dashed, very thick] (v1) -- (v15);
\end{tikzpicture}
    \end{subfigure}
    
    \caption{Illustration of the construction for $t = 4$ and $t = 5$ ($13$ and $19$ vertices respectively). Dashed red edges are not part of the construction, and adding them introduces the $K_4$ subgraph given by the green vertices and already present thick edges.}
    \label{fig:construction}
\end{figure}

An interesting lingering question is whether the construction in \Cref{thm-r4t} is as small as possible. Our SAT experiments demonstrate this for $t \le 5$.

\begin{question} \label{q-minimal}
    For each $t \ge 4$, is the smallest doubly saturated $R(4,t)$-good graph of size $6t-11$?
\end{question}

We remark that the largest $R(4,t)$-good graphs are known to be of size $\widetilde{\Theta}(t^3)$~\cite{r4t-lower,r4t-upper}, so the doubly saturated $R(4,t)$-good graphs constructed in \Cref{thm-r4t} are much smaller than the largest $R(4,t)$-good graphs.

\section{Other families of doubly saturated Ramsey graphs} \label{sec-other-families}

\subsection{Doubly saturated $R(3,t)$-good graphs} \label{sec-r3t}

We have some experimental results regarding doubly saturated $R(3,t)$-good graphs, including a conjectural infinite family. First, we present experimental observations for small values of $t$:
\begin{itemize}
    \item There is a unique doubly saturated $R(3,3)$-good graph, namely $C_5$.
    \item There are no doubly saturated $R(3,4)$-good graphs.
    \item There is a unique doubly saturated $R(3,5)$-good graph, namely the circulant graph on 13 vertices with distances $\{1,5\}$.
    \item There are no doubly saturated $R(3,6)$-good graphs.
    \item Using our SAT encoding, we found that there is a doubly saturated $R(3,7)$-good graph on 20 vertices, and there are none with fewer than 20 vertices. We found 18 satisfying assignments that were all isomorphic, which suggests that the solution may be unique. Interestingly, the graph is 5-regular and vertex-transitive, although it is not circulant. We asked ChatGPT whether it was possible to decompose the graph into various patterned substructures, and with its help, we eventually discovered that it can be partitioned into four disjoint 5-cycles with a matching between each pair of cycles (see \Cref{fig:r37}). This pattern might form the basis of a general construction.
    \item Using our SAT encoding, we have found two non-isomorphic doubly saturated $R(3,8)$-good graphs on 25 vertices. There may be more than two such graphs, and we do not know if they are minimal, although we do know that there are none on 21 or fewer vertices.
    \item There is a unique $R(3,9)$-good graph on 35 vertices~\cite{r39-unique}, and it is therefore doubly saturated. It is also a circulant graph.
    \item We have found over 100 doubly saturated $R(3,10)$-good graphs on 39 vertices by searching through the known $R(3,10)$-good graphs on 39 vertices.
\end{itemize}

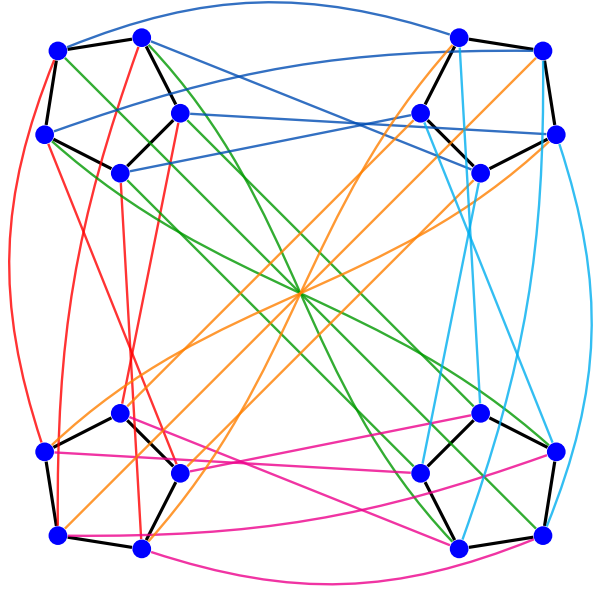
\begin{figure}
    \centering
    \begin{tikzpicture}[rotate=45,
    scale=0.8,
  vertex/.style={circle, fill=blue, minimum size=0mm, inner sep=2.5pt},
  cycleedge/.style={black, line width=1.2pt},
  matchedge/.style={line width=0.9pt, opacity=0.8}
]

\node[vertex] (1) at (0.000000, 5.700000) {~};
\node[vertex] (2) at (-5.700000, 0.000000) {~};
\node[vertex] (3) at (-1.141268, -4.870820) {~};
\node[vertex] (4) at (1.141268, 4.870820) {~};
\node[vertex] (5) at (-0.000000, -5.700000) {~};
\node[vertex] (6) at (5.700000, 0.000000) {~};
\node[vertex] (7) at (-4.870820, 1.141268) {~};
\node[vertex] (8) at (4.870820, -1.141268) {~};
\node[vertex] (9) at (4.870820, 1.141268) {~};
\node[vertex] (10) at (1.141268, -4.870820) {~};
\node[vertex] (11) at (-3.529180, 0.705342) {~};
\node[vertex] (12) at (3.529180, -0.705342) {~};
\node[vertex] (13) at (-3.529180, -0.705342) {~};
\node[vertex] (14) at (0.705342, -3.529180) {~};
\node[vertex] (15) at (-1.141268, 4.870820) {~};
\node[vertex] (16) at (-4.870820, -1.141268) {~};
\node[vertex] (17) at (-0.705342, -3.529180) {~};
\node[vertex] (18) at (0.705342, 3.529180) {~};
\node[vertex] (19) at (3.529180, 0.705342) {~};
\node[vertex] (20) at (-0.705342, 3.529180) {~};

\draw[cycleedge] (1)--(4)--(18)--(20)--(15)--(1);
\draw[cycleedge] (6)--(8)--(12)--(19)--(9)--(6);
\draw[cycleedge] (5)--(3)--(17)--(14)--(10)--(5);
\draw[cycleedge] (2)--(7)--(11)--(13)--(16)--(2);


\draw[matchedge,red] (1) to[bend right=20] (7);
\draw[matchedge,red] (4)to[bend right=10](2);
\draw[matchedge,red] (15)--(13);
\draw[matchedge,red] (18)--(11);
\draw[matchedge,red] (20)--(16);
\draw[matchedge,green!60!black] (1)--(5);
\draw[matchedge,green!60!black] (4) to[out=-90,in=90] (3);
\draw[matchedge,green!60!black] (15) to[out=-85,in=95] (10);
\draw[matchedge,green!60!black] (18)--(14);
\draw[matchedge,green!60!black] (20)--(17);
\draw[matchedge,blue!70!green] (1) to[bend left=20] (9);
\draw[matchedge,blue!70!green] (4) -- (12);
\draw[matchedge,blue!70!green] (15) to[bend left=10] (6);
\draw[matchedge,blue!70!green] (18)--(8);
\draw[matchedge,blue!70!green] (20)--(19);
\draw[matchedge,magenta] (2) to[bend right=10] (10);
\draw[matchedge,magenta] (7)--(17);
\draw[matchedge,magenta] (11)--(3);
\draw[matchedge,magenta] (13)--(14);
\draw[matchedge,magenta] (16)to[bend right=20](5);
\draw[matchedge,orange] (2) -- (6);
\draw[matchedge,orange] (7) to[out=0,in=-180] (8);
\draw[matchedge,orange] (11)--(19);
\draw[matchedge,orange] (13) -- (12);
\draw[matchedge,orange] (16) to[out=5,in=-175] (9);
\draw[matchedge,cyan] (3) to[bend right=10] (6);
\draw[matchedge,cyan] (5) to[bend right=20](8);
\draw[matchedge,cyan] (10)--(19);
\draw[matchedge,cyan] (14)--(9);
\draw[matchedge,cyan] (17)--(12);
\end{tikzpicture}
    \caption{A doubly saturated $R(3,7)$-good graph on 20 vertices}
    \label{fig:r37}
\end{figure}

Using a similar workflow as described in \Cref{sec-r4t}, we collaborated with ChatGPT-5.1 Pro to find circulant constructions for doubly saturated $R(3,t)$-good graphs. After much experimentation, we conjecture the following:
\begin{conjecture} \label{conj-r3t}
    For every odd $t \ge 17$, the following graph is doubly saturated $R(3,t)$-good: the circulant graph on $5t-10$ vertices with distances $[t-4,t-3] \cup [t+1, (3t-9)/2] \cup \{(3t-5)/2\} \cup \{2t-4\}$.
\end{conjecture}
We have not yet managed to prove this conjecture, although we have checked it for $t \le 63$. We expect that in the near future, it will be feasible for LLMs to both prove \Cref{conj-r3t} and formally verify it as was done for \Cref{thm-r4t}.

\subsection{Doubly saturated $R(s,s)$-good graphs} \label{sec-rss}

The question of whether there is a doubly saturated $R(s,s)$-good graph is perhaps the most interesting case of \Cref{conj-ds}.  We have verified this claim for $s \le 20$ using Paley graphs of prime order. Since Paley graphs are symmetric and self-complementary, to check whether a Paley graph is doubly saturated, we only need to add an arbitrary edge and check if the clique number increases. We use Cliquer~\cite{cliquer} to compute clique numbers of these graphs. \Cref{tab:paley} shows, for each $s \in [3,20]$, the smallest prime $p$ such that the Paley graph of order $p$ is doubly saturated $R(s,s)$-good.

\begin{table}
    \centering
    \caption{The smallest prime $p$ such that the Paley graph on $p$ vertices is a doubly saturated $R(s,s)$-good graph}
    {
    \setlength{\tabcolsep}{2pt} 
    \begin{tabular}{ c | c | c | c | c | c | c | c | c | c | c | c | c | c | c | c | c | c | c }
    \toprule
  ~~$s$~~ & 3 & 4 & 5 & 6 & 7 & 8 & 9 & 10 & 11 & 12 & 13 & 14 & 15 & 16 & 17 & 18 & 19 & 20 \\
  \midrule
  ~~\small $p$~~ & \small 5 & \small 13 & \small 29 & \small 53 & \small 97 & \small 137 & \small 173 & \small 317 & \small 577 & \small 541 & \small 613 & \small 1301 & \small 1373 & \small 1481 & \small 2389 & \small 3001 & \small 4861 & \small 3169\\
  \bottomrule
    \end{tabular}
    }
    \label{tab:paley}
\end{table}

As depicted in \Cref{fig:paley-prop}, our experimental results show that approximately 75\% of prime-order Paley graphs are doubly saturated. On the basis of these results, we strongly expect that there is a doubly saturated $R(s,s)$-good Paley graph for each $s \ge 3$, although this seems difficult to prove.

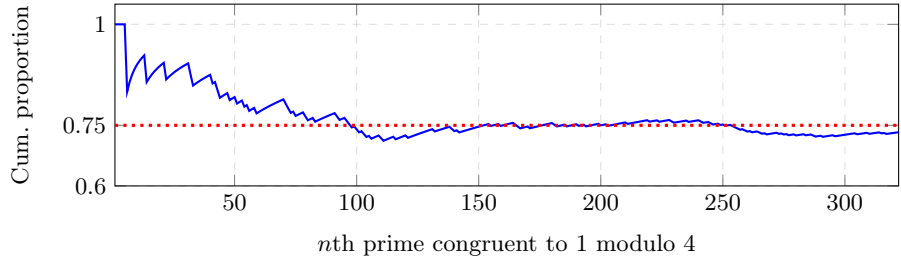
\begin{figure}
    \centering
    \begin{tikzpicture}
\begin{axis}[
    width=12cm,
    height=4cm,
    xlabel={$n$th prime congruent to 1 modulo 4},
    ylabel={Cum. proportion},
    ymin=0.6, ymax=1.05,
    xmin=1, xmax=322,
    ytick={0.6, 0.75, 1}, 
    grid=both,
    grid style={dashed, opacity=0.5},
]

\addplot[
    color=blue,
    thick
] coordinates {
    (1,1.0000) (2,1.0000) (3,1.0000) (4,1.0000) (5,1.0000) (6,0.8333) (7,0.8571) (8,0.8750) (9,0.8889) (10,0.9000) (11,0.9091) (12,0.9167) (13,0.9231) (14,0.8571) (15,0.8667) (16,0.8750) (17,0.8824) (18,0.8889) (19,0.8947) (20,0.9000) (21,0.9048) (22,0.8636) (23,0.8696) (24,0.8750) (25,0.8800) (26,0.8846) (27,0.8889) (28,0.8929) (29,0.8966) (30,0.9000) (31,0.9032) (32,0.8750) (33,0.8485) (34,0.8529) (35,0.8571) (36,0.8611) (37,0.8649) (38,0.8684) (39,0.8718) (40,0.8750) (41,0.8537) (42,0.8571) (43,0.8372) (44,0.8182) (45,0.8222) (46,0.8261) (47,0.8298) (48,0.8125) (49,0.8163) (50,0.8200) (51,0.8039) (52,0.8077) (53,0.8113) (54,0.7963) (55,0.8000) (56,0.7857) (57,0.7895) (58,0.7931) (59,0.7797) (60,0.7833) (61,0.7869) (62,0.7903) (63,0.7937) (64,0.7969) (65,0.8000) (66,0.8030) (67,0.8060) (68,0.8088) (69,0.8116) (70,0.8143) (71,0.8028) (72,0.7917) (73,0.7808) (74,0.7838) (75,0.7733) (76,0.7763) (77,0.7792) (78,0.7821) (79,0.7722) (80,0.7625) (81,0.7654) (82,0.7683) (83,0.7590) (84,0.7619) (85,0.7647) (86,0.7674) (87,0.7701) (88,0.7727) (89,0.7753) (90,0.7778) (91,0.7802) (92,0.7717) (93,0.7634) (94,0.7660) (95,0.7684) (96,0.7604) (97,0.7526) (98,0.7449) (99,0.7475) (100,0.7400) (101,0.7327) (102,0.7353) (103,0.7282) (104,0.7212) (105,0.7238) (106,0.7170) (107,0.7196) (108,0.7222) (109,0.7248) (110,0.7182) (111,0.7117) (112,0.7143) (113,0.7168) (114,0.7193) (115,0.7217) (116,0.7155) (117,0.7179) (118,0.7203) (119,0.7227) (120,0.7250) (121,0.7190) (122,0.7213) (123,0.7236) (124,0.7258) (125,0.7280) (126,0.7302) (127,0.7323) (128,0.7344) (129,0.7364) (130,0.7385) (131,0.7405) (132,0.7348) (133,0.7368) (134,0.7388) (135,0.7407) (136,0.7426) (137,0.7445) (138,0.7464) (139,0.7410) (140,0.7357) (141,0.7376) (142,0.7324) (143,0.7343) (144,0.7361) (145,0.7379) (146,0.7397) (147,0.7415) (148,0.7432) (149,0.7450) (150,0.7467) (151,0.7483) (152,0.7500) (153,0.7516) (154,0.7532) (155,0.7484) (156,0.7500) (157,0.7516) (158,0.7532) (159,0.7484) (160,0.7500) (161,0.7516) (162,0.7531) (163,0.7546) (164,0.7561) (165,0.7515) (166,0.7470) (167,0.7425) (168,0.7440) (169,0.7456) (170,0.7471) (171,0.7427) (172,0.7442) (173,0.7457) (174,0.7471) (175,0.7486) (176,0.7500) (177,0.7514) (178,0.7528) (179,0.7542) (180,0.7556) (181,0.7514) (182,0.7473) (183,0.7486) (184,0.7500) (185,0.7514) (186,0.7473) (187,0.7487) (188,0.7500) (189,0.7513) (190,0.7526) (191,0.7487) (192,0.7500) (193,0.7513) (194,0.7474) (195,0.7487) (196,0.7500) (197,0.7513) (198,0.7525) (199,0.7538) (200,0.7500) (201,0.7512) (202,0.7525) (203,0.7488) (204,0.7500) (205,0.7512) (206,0.7524) (207,0.7536) (208,0.7548) (209,0.7560) (210,0.7571) (211,0.7536) (212,0.7547) (213,0.7559) (214,0.7570) (215,0.7581) (216,0.7593) (217,0.7604) (218,0.7615) (219,0.7626) (220,0.7591) (221,0.7602) (222,0.7613) (223,0.7623) (224,0.7589) (225,0.7600) (226,0.7611) (227,0.7621) (228,0.7632) (229,0.7598) (230,0.7565) (231,0.7576) (232,0.7586) (233,0.7597) (234,0.7607) (235,0.7617) (236,0.7585) (237,0.7595) (238,0.7605) (239,0.7615) (240,0.7625) (241,0.7593) (242,0.7562) (243,0.7572) (244,0.7582) (245,0.7551) (246,0.7520) (247,0.7530) (248,0.7540) (249,0.7510) (250,0.7520) (251,0.7490) (252,0.7500) (253,0.7510) (254,0.7480) (255,0.7451) (256,0.7422) (257,0.7393) (258,0.7403) (259,0.7375) (260,0.7346) (261,0.7356) (262,0.7328) (263,0.7338) (264,0.7348) (265,0.7321) (266,0.7293) (267,0.7303) (268,0.7276) (269,0.7286) (270,0.7296) (271,0.7306) (272,0.7316) (273,0.7289) (274,0.7299) (275,0.7273) (276,0.7283) (277,0.7256) (278,0.7266) (279,0.7276) (280,0.7250) (281,0.7260) (282,0.7270) (283,0.7279) (284,0.7254) (285,0.7263) (286,0.7273) (287,0.7247) (288,0.7222) (289,0.7232) (290,0.7241) (291,0.7216) (292,0.7226) (293,0.7235) (294,0.7245) (295,0.7254) (296,0.7230) (297,0.7239) (298,0.7248) (299,0.7258) (300,0.7267) (301,0.7276) (302,0.7285) (303,0.7294) (304,0.7303) (305,0.7279) (306,0.7288) (307,0.7296) (308,0.7305) (309,0.7314) (310,0.7290) (311,0.7299) (312,0.7308) (313,0.7316) (314,0.7325) (315,0.7302) (316,0.7310) (317,0.7287) (318,0.7296) (319,0.7304) (320,0.7312) (321,0.7321) (322,0.7329)
};

\addplot[
    color=red,
    dotted,
    very thick,
    domain=1:322,
    samples=2 
] {0.75};

\end{axis}
\end{tikzpicture}
    \caption{Cumulative proportion of prime-order Paley graphs that are doubly saturated}
    \label{fig:paley-prop}
\end{figure}

\section{A lower bound for the double saturation number} \label{sec-lb}

We define the \emph{double saturation number} of $s$ and $t$ as follows:
\[
    \DS(s,t) := \inf \{n \mid \text{there is a doubly saturated $R(s,t)$-good graph on $n$ vertices}\}.
\]
\Cref{q-minimal} asks whether $\DS(4,t) = 6t-11$ for all $t \ge 4$. More generally, it is natural to ask for bounds on $\DS(s,t)$. We asked Gemini 3.1 Deep Think to prove a lower bound on $\DS(s,t)$, and it proved that $\DS(s,t) \ge 2s + 2t - 7$ with an elegant argument. To present the argument, we need a theorem due to Hajnal~\cite{hajnal}. We say that a graph $G$ is \emph{$k$-saturated} if it does not contain a $K_{k+1}$ subgraph, but adding any edge to $G$ yields a graph containing a $K_{k+1}$. Hajnal proved the following:
\begin{theorem}[{\!\!\!\!\;\;\!\!\cite[Theorem~1]{hajnal}}] \label{thm-hajnal}
    If $G$ is $k$-saturated, then either $G$ has a vertex adjacent to every other vertex or $\delta(G) \ge 2(k-1)$.
\end{theorem}
\noindent
Now, we can prove our lower bound on $\DS(s,t)$:
\begin{theorem}
    For all $s,t \ge 3$, we have $\DS(s,t) \ge 2s + 2t - 7$.
\end{theorem}
\begin{proof}
    Suppose that $G$ is a doubly saturated $R(s,t)$-good graph. First, we claim that $G$ does not have a vertex adjacent to every other vertex. Indeed, suppose for the sake of contradiction that $v \in V(G)$ is adjacent to every vertex in $V(G) \setminus \{v\}$. Let $w \in V(G) \setminus \{v\}$ be an arbitrary vertex, and let $G'$ be the graph obtained from $G$ by removing the edge $\{v,w\}$. Then, $G'$ contains an $I_t$ involving $v$. Since $v$ is adjacent to every other vertex except $w$ in $G'$, the largest independent set involving $v$ is $\{v,w\}$, which contradicts our assumption that $t \ge 3$. Thus, $G$ does not have a vertex adjacent to every other vertex. By symmetry, the same property holds for $\overline{G}$.

    Now, by \Cref{thm-hajnal}, we have $\delta(G) \ge 2(s-2)$ and $\delta(\overline{G}) \ge 2(t-2)$. Thus, $\Delta(G) = n-1 - \delta(\overline{G}) \le n-1 - 2(t-2) = n-2t+3$. Since $\delta(G) \le \Delta(G)$, we have $2(s-2) \le n-2t+3$, so $n \ge 2s+2t-7$.
\end{proof}

\section{Conclusion}

We used SAT solvers and LLMs to study doubly saturated Ramsey-good graphs. Our work suggests several further directions related to these graphs:
\begin{itemize}
    \item \Cref{conj-ds} is the main open problem raised by our work. Finding new infinite families of doubly saturated Ramsey-good graphs would constitute partial progress toward proving this conjecture and has the potential to uncover intriguing families of graphs. Indeed, there is reason to expect such graphs to have nice combinatorial or algebraic properties, as illustrated by the infinite families from \Cref{thm-r4t} and \Cref{conj-r3t}, the graph depicted in \Cref{fig:r37}, and the Paley graphs discussed in \Cref{sec-rss}.
    \item Another further direction is to better understand the behavior $\DS(s,t)$. In \Cref{q-minimal}, we asked whether $\DS(4,t) = 6t-11$. We are also interested in whether, for each fixed $s \ge 3$, the growth of $\DS(s,t)$ is linear; that is, is $\DS(s,t) = O_s(t)$? \Cref{thm-r4t} and \Cref{conj-r3t} provide some limited evidence in favor of this hypothesis. It would also be interesting to understand the asymptotics of $\DS(s,s)$. As a concrete question, is $\DS(s,s) = s^{O(1)}$?
    \item Given arbitrary graphs $H_1$ and $H_2$, one can consider graphs $G$ such that (i) $G$ does not contain $H_1$ as a subgraph, (ii) $\overline{G}$ does not contain $H_2$ as a subgraph, (iii) adding any edge to $G$ yields a graph that contains $H_1$ as a subgraph, and (iv) adding any edge to $\overline{G}$ yields a graph that contains $H_2$ as a subgraph. Taking $H_1 = K_s$ and $H_2 = K_t$ yields the doubly saturated $R(s,t)$-good graphs that we have been studying, although other choices of graphs may be equally interesting. The study of graphs satisfying conditions (i) and (ii) is a central topic in Ramsey theory~\cite{ramsey-subgraph,ramsey-subgraph-2,ramsey-subgraph-3,ramsey-subgraph-4}.
\end{itemize}

More generally, our work creates an impetus to find other types of mathematical problems to which our computer-assisted methodology can be fruitfully applied. The features that made this problem amenable to computer-assistance were the following:
\begin{itemize}
    \item Data for small cases are very helpful for understanding the problem while being hard to compute by hand, which makes SAT solvers and other computational experiments useful.
    \item There is a vast array of techniques that might be applicable to the problem, which makes leveraging the encyclopedic knowledge of LLMs powerful.
    \item Proofs are likely to involve relatively elementary combinatorial arguments, which allows them to be discovered and formalized by LLMs.
\end{itemize}
With the mathematical abilities of LLMs rapidly advancing, we expect that the last feature will become increasingly unnecessary.

\begin{credits}
\subsubsection{\ackname} This research is supported by the DARPA expMath program through the DARPA CMO contract number HR0011262E028. Przybocki was additionally supported by the NSF Graduate Research Fellowship Program under Grant No. DGE-2140739.

\subsubsection{\discintname}
The authors have no competing interests to declare that are
relevant to the content of this article.
\end{credits}
%
%
\bibliographystyle{splncs04}
\bibliography{bib}

\end{document}